\documentclass[10pt]{amsart}
\usepackage[margin=1in]{geometry}
\usepackage{mathrsfs}
\usepackage{amsmath, enumerate, graphicx, float}
\usepackage[colorlinks = true,
linkcolor = blue,
urlcolor  = blue,
citecolor = blue,
anchorcolor = blue]{hyperref}
\usepackage{amsfonts}
\usepackage{amssymb, indentfirst}
\usepackage[toc,page]{appendix}
\usepackage{tikz}
\usetikzlibrary{matrix,arrows}
\tikzset{global scale/.style={
    scale=#1,
    every node/.style={scale=#1}
  }
}
\usetikzlibrary{patterns}
\usepackage[noadjust]{cite}

\def\CC{\mathbb{C}}

\def\NN{\mathbb{N}}
\def\PP{\mathbb{P}}

\def\RR{\mathbb{R}}

\def\ZZ{\mathbb{Z}}

\def\shF{\mathscr{F}}

\def\shI{\mathcal{I}}

\DeclareMathOperator\conv{conv}

\DeclareMathOperator\Vol{Vol}
\DeclareMathOperator\reg{reg}

\DeclareMathOperator\codim{codim}
\DeclareMathOperator\ehr{ehr}
\DeclareMathOperator\Ehr{Ehr}
\DeclareMathOperator\h{H}

\newtheorem{theorem}{Theorem}[section]
\newtheorem{corollary}[theorem]{Corollary}
\newtheorem{proposition}[theorem]{Proposition}
\newtheorem{lemma}[theorem]{Lemma}

\theoremstyle{remark}
\newtheorem{remark}[theorem]{Remark}
\newtheorem{example}[theorem]{Example}
\theoremstyle{definition}
\newtheorem{definition}[theorem]{Definition}

\usepackage{tikz-3dplot}

\begin{document}
\title[An Eisenbud-Goto-type Upper Bound for Weighted Projective Spaces]{An Eisenbud-Goto-type Upper Bound for the Castelnuovo-Mumford Regularity of Fake Weighted Projective Spaces}
\author{Bach Le Tran}
\email{b.tran@sms.ed.ac.uk}
\address{School of Mathematics\\ University of Edinburgh\\ James Clerk Maxwell Building\\ Peter Guthrie Tait Road\\ Edinburgh EH9 3FD}
\maketitle
\begin{abstract}	
	We will give an upper bound for the $k$-normality of very ample lattice simplices, and then give an Eisenbud-Goto-type bound for some special classes of projective toric varieties.
\end{abstract}
\section{Introduction}
	The study of the Castelnuovo-Mumford regularity for projective varieties has been greatly movitated by the Eisenbud-Goto conjecture (\cite{Eisenbud1984}) which asks for any irreducible and reduced variety $X$, is it always the case that
	\begin{equation}\label{Eisenbud-Goto conjecture}
		\reg(X)\le \deg(X)-\codim(X)+1?
	\end{equation}
	
	The Eisenbud-Goto conjecture is known to be true for some particular cases. For example, it holds for smooth surfaces in characteristic zero (\cite{Lazarsfeld1987}), connected reduced curves (\cite{Giaimo2005}), etc. Inspired by the conjecture, there are also many attempts to give an upper bound for the Castelnuovo-Mumford regularity for various types of algebraic and geometric structures (\cite{Sturmfels1995}, \cite{Kwak1998}, \cite{Miyazaki2000}, \cite{Derksen2002},  etc).  
	
	For toric geometry, suppose that $(X,L)$ is a polarized projective toric varieties such that $L$ is very ample. Then there is a corresponding very ample lattice polytope $P:=P_L$ associated to $L$ such that $\Gamma(X, L)=\bigoplus_{m\in P\cap M}\CC\cdot \chi^m$ (\cite[Section 5.4]{Cox2011}). Therefore, by studying the $k$-normality of $P$ (cf. Definition \ref{k-normality definition}), we can obtain the $k$-normality and also the regularity of the original variety $(X,L)$. For the purpose of this article, we will focus on the case that $X$ is a fake weighted projective $d$-space and $P_L$ a $d$-simplex.
	
	For any fake weighted projective $d$-space $X$ embedded in $\PP^r$ via a very ample line bundle, Ogata (\cite{Ogata2005}) gives an upper bound for its $k$-normality: 
	\[k_X\le \dim X+\left\lfloor \frac{\dim X}{2}\right\rfloor-1.\]
	In this article, we will improve Ogata's bound by giving a new upper bound for the $k$-normality of very ample lattice simplices and show that 
	
\begin{equation}\label{EG style bound}
	\reg(X)\le \deg(X)-\codim(X)+\left\lfloor\frac{\dim X}{2}\right\rfloor.
\end{equation}

Recently, McCullough and Peeva showed some counterexamples to the Eisenbud-Goto conjecture and that the difference $\reg(X)-\deg(X)+\codim(X)$ can be arbitrary large (\cite[Counterexample 1.8]{McCullough2017}). However, for any fake weighted projective space $X$ embedded in $\PP^r$ via a very ample line bundle, it follows from \eqref{EG style bound} that $\reg(X)-\deg(X)+\codim(X)$ is bounded above by $\dim(X)/2$. Furthermore, we will show that the Eisenbud-Goto conjecture holds for any projective toric variety corresponding to a very ample Fano simplex.

\subsection*{Acknowledgments}
We would like to thank Milena Hering for reading the drafts of this article and for some valuable suggestions.
\section{Background Material}
\subsection{Toric Varieties and Lattice Simplices}
We begin this section by recalling the definition of the Castelnuovo-Mumford regularity:
\begin{definition}
	Let $X\subseteq \PP^r$ be an irreducible projective variety and $\shF$ a coherent sheaf over $X$. We say that $\shF$ is $k$-regular if
	\[\h^i(X,\shF(k-i))=0\]
	for all $i>0$. 
	The regularity of $\shF$, denoted by $\reg(\shF)$, is the minimum number $k$ such that $\shF$ is $k$-regular. We say that $X$ is $k$-regular if the ideal sheaf $\shI_X$ of $X$ is $k$-regular and use $\reg(X)$ to denote the regularity of $X$ (or of $\shI_X$). 
\end{definition}

As the main object of the article is to find an upper bound for $k$-normality of very ample lattice simplices, it is important for us to revisit the definition of $k$-normality of lattice polytopes.
\begin{definition}\label{k-normality definition}
	A lattice polytope $P$ is $k$-normal if the map
	\[\underbrace{P\cap M+\cdots+P\cap M}_{k\text{ times}}\rightarrow kP\cap M\]
	is surjective. The $k$-normality of $P$, denoted by $k_P$, is the smallest positive integer $k_P$ such that $P$ is $k$-normal for all $k\ge k_P$. 
\end{definition}

Suppose now that $X$ is a fake weighted projective $d$-space embedded in $\PP^r$ via a very ample line bundle. Then the polytope $P$ corresponding to the embedding is a very ample lattice $d$-simplex. Furthermore, $\codim(X)=|P\cap M|-(d+1)$, where $M$ is the ambient lattice, and $\deg(X)=\Vol(P)$, the normalized volume of $P$. 



We have a combinatorial interpretation of $\reg(X)$ in terms of $k_P$ and $\deg (P)$ (\cite[Proposition 4.1.5]{Tran2018}) as follows:
\begin{equation}\label{reg(X) formula}
	\reg(X)=\max\{k_P, \deg (P)\}+1.
\end{equation}
From this, we obtain a combinatorial form of the Eisenbud-Goto conjecture: for very ample lattice polytope $P\subset M_{\RR}$, is it always true that
\[\max\{\deg (P), k_P\}\le \Vol(P)-|P\cap M|+d+1?\]
The first inequality was confirmed to be true recently (\cite[Proposition 2.2]{Hofscheier2017}); namely,
\begin{equation}\label{HKN}
	\deg(P)\le \Vol(P)-|P\cap M|+d+1.
\end{equation}
Therefore, in order to verify the Eisenbud-Goto conjecture for the polarized toric variety $(X,L)$, it suffices to check if
\begin{equation}\label{k_P inequality}
	k_P\le \Vol(P)-|P\cap M|+d+1.
\end{equation}

\subsection{Ehrhart Theory}
We now recall some basic facts about Ehrhart theory of polytopes and the definition of their degree.

Let $P$ be a lattice polytope of dimension $d$. We define $\ehr_P(k)=|kP\cap M|$, the number of lattice points in $kP$. Then from Ehrhart's theory (\cite{Ehrhart1962, Stanley1980}),
\[\Ehr_P(t)=\sum_{k=0}^{\infty}\ehr_P(k)t^k=\frac{h_P^*(t)}{(1-t)^{d+1}}\]
for some polynomial $h_P^*\in\ZZ_{\ge 0}[t]$ of degree less than or equal to $d$. Let $h_P^*(t)=\sum_{i=0}^dh_i^*t^i$. We have
\[h_0^*=1, \hspace*{2mm} h_1^*=|P\cap M|-d-1,  \hspace*{2mm} h_d^*=|P^0\cap M|, \text{ and }\sum_{i=0}^dh_i^*=\Vol(P). \]


\begin{definition}[{\cite[Remark 2.6]{Batyrev2007}}]\label{degree of polytope}
	Let $P$ be a lattice polytope of dimension $d$. We define the degree of $P$\index{polytope!degree}, denoted by $\deg(P)$, to be the degree of $h_P^*(t)$. Equivalently, 
	\[
	\deg(P)=\begin{cases}
	d & \text{if } |P^0\cap M|\neq 0.\\
	\min\left\{i\in\ZZ_{\ge 0}|(kP)^0\cap M=\emptyset\text{ for all }1\le k\le d-i\right\} &\text{otherwise}. 
	\end{cases}
	\]
\end{definition}

\section{$k$-normality of Very Ample Simplices}

The following lemma by Ogata is crucial to the main result of this article:

\begin{lemma}[{\cite[Lemma 2.1]{Ogata2005}}]\label{Ogata05, Lemma 2.1}
	Let $P=\conv(v_0,\ldots, v_d)$ be a very ample lattice $n$-simplex. Suppose that $k\ge 1$ is an integer and $x\in kP\cap M$. For any $i=0,\ldots, d$, we have
	\[x+(k-1)v_i =\sum_{j=1}^{2k-1}u_j\]
	for some $u_j\in P\cap M$.
\end{lemma}

Using the ideas in \cite[Lemma 2.5]{Ogata2005}, we generalize the above lemma as follows.

\begin{lemma}\label{Modified Ogata 05, Lemma 2.5}
	Suppose that $P=\conv(v_0,\ldots,v_d)$ is a very ample $d$-simplex. Let $k\in \NN_{\ge 1}$. Then for any $x\in kP\cap M$, $a_0,\ldots, a_{d}\in\ZZ_{\ge 0}$ such that $\sum_{i=0}^{d}a_i=k-1$, we have
	\[\sum_{i=0}^{d}a_iv_i+x=\sum_{i=1}^{2k-1}u_i\]
	for some $u_i\in P\cap M$.
\end{lemma}
\begin{proof}
	We will use induction in this proof. The case $k=1$ is trivial. Suppose that the lemma holds for $k=s-1$. We will now show that it holds for $k=s$; i.e., for any $x\in sP\cap M$, $a_1,\ldots, a_{d}\in\ZZ_{\ge 0}$ such that $\sum_{i=0}^{d}a_i=s-1$, we have
\begin{equation}\label{inductive condition}
\sum_{i=0}^{d}a_iv_i+x=\sum_{i=1}^{2s-1}u_i
\end{equation}
	for some $u_i\in P\cap M$. Without loss of generality, we can take $a_0$ to be positive and move $v_0$ to the origin. By Lemma \ref{Ogata05, Lemma 2.1}, 
	\[(s-1)v_0+x=\sum_{i=1}^{2s-1}w_i\]
	for some $w_i\in P\cap M$. Since $v_0=0$, we can write $x=\sum_{i=1}^{2s-1}w_i$. If $w_i+w_j\in P\cap M$ for any $i\neq j$, then we can let $t_i=w_{2i-1}+w_{2i}$ for $i=1,\ldots, s-1$ and have $x=t_1+\cdots+t_{s-1}+w_{2s-1}$, which lies in $\sum_{i=1}^sP\cap M$. Therefore,
\[\sum_{i=0}^{d}a_iv_i+x=\sum_{i=0}^{d}a_iv_i+\sum_{i=1}^{s-1}t_i+w_{2s-1},\]
which satisfies \eqref{inductive condition}. Conversely, without loss of generality, suppose that $w_1+w_2\notin P\cap M$. Then since $x=w_1+w_2+(w_3+\cdots+w_{2s-1})\in sP\cap M$, we have $y:=w_3+\cdots+w_{2s-1}\in(s-1)P\cap M$ and $v_0+x=w_1+w_2+y$. Using the induction hypothesis,
	\[\underbrace{(a_0-1)v_0+\sum_{i=1}^{d}a_iv_i}_{a_0 - 1 + \sum_{i=1}^da_i = s-2}+y=\sum_{i=1}^{2(s-1)-1}w'_i\]
	for some $w'_i\in P\cap M$. It follows that
	\begin{align*}
	\sum_{i=0}^{d}a_iv_i+x &=v_0+x+(a_0-1)v_0+\sum_{i=1}^{d}a_iv_i\\
	&= w_1+w_2+y+(a_0-1)v_0+\sum_{i=1}^{d}a_iv_i\\
	&= w_1+w_2+\sum_{i=0}^{2(s-1)-1}w'_i.
	\end{align*}
	The conclusion follows.
\end{proof}

Now define the invariants $d_P$ and $\nu_P$ as in \cite[Definition 2.2.8]{Tran2018}:
\begin{definition}
	Let $P$ be a lattice polytope with the set of vertices $\mathcal{V}=\{v_0,\ldots, v_{n-1}\}$. We define $d_P$ to be the smallest positive integer such that for every $k\ge d_P$,
	\[(k+1)P\cap M=P\cap M+kP\cap M.\]
	We also define $\nu_P$ to be the smallest positive integer such that for any $k\ge \nu_P$,
	\[(k+1)P\cap M=\mathcal{V}+kP\cap M.\]
	Notice that for $P$ an $n$-simplex, $d_P\le \nu_P\le n-1$.
\end{definition}

\begin{proposition}\label{Bound of k_P for very ample simplex}
	Let $P=\conv(v_0,\ldots,v_d)$ be a very ample $d$-simplex. Then 
	\[k_P\le \nu_P+d_P-1.\]
\end{proposition}
\begin{proof}
	For any $k\ge d_P+\nu_P-1$ and $p\in kP\cap M$, by the definition of $d_P$ and $\nu_P$, we have
	\begin{equation}\label{first presentation of p}
	p=x+\sum_{i=1}^{\nu_P-d_P}u_i+\sum_{i=0}^{d}a_iv_i
	\end{equation}
	for some $x\in d_PP\cap M$, $u_i\in P\cap M$, $\sum_{i=0}^{d}a_i=k-\nu_P$. By assumption, $k-\nu_P\ge d_P-1$, so it follows from Lemma \ref{Modified Ogata 05, Lemma 2.5} that
	\begin{equation}\label{linear presentation of x+a_iv_i}
	x+\sum_{i=0}^{d}a_iv_i=\sum_{i=1}^{d_P+k-\nu_P}u'_i
	\end{equation}
	for some $u'_i\in P\cap M$. Substitute \eqref{linear presentation of x+a_iv_i} into \eqref{first presentation of p}, we have
	\[p=\sum_{i=1}^{\nu_P-d_P}u_i+\sum_{i=1}^{d_P+k-\nu_P}u'_i.\]
	The conclusion follows.
\end{proof}

\begin{remark}
	This bound is stronger than \cite[Proposition 2.4]{Ogata2005} since $\nu_P\le d$ (\cite[Proposition 2.2]{Tran2018}) and $d_P\le d/2$ (\cite[Proposition 2.2]{Ogata2005}).
\end{remark}

\section{An Eisenbud-Goto-Type Upper Bound for Very Ample Simplices}

Suppose that $P$ is a very ample simplex. If $P$ is unimodularly equivalent to the standard simplex $\Delta_d=\conv(0,e_1,\ldots,e_d)$ then \eqref{k_P inequality} holds. Now consider the case $P$ is not unimodularly equivalent to $\Delta_d$.

The following lemma is a rewording of \cite[Proposition IV.10]{Hering2006} to fit our purpose. We provide a proof for the sake of completeness. 
\begin{lemma}\label{nu_P<=deg P-1}
	Let $\mathcal{V}=\{v_0,\ldots,v_d\}$ and suppose that $P=\conv(\mathcal{V})$ is a lattice simplex not unimodularly equivalent to $\Delta_d$. Then $\deg(P)\ge \nu_P$.
\end{lemma}
\begin{proof}
	Since $\nu_P\le d$, it suffices to show that for any $d\ge k\ge \deg (P)$,
	\[\mathcal{V}+kP\cap M\twoheadrightarrow (k+1)P\cap M.\]
	Indeed, any $x\in (k+1)P\cap M$ can be written as $x=\sum_{i=0}^da_iv_i$
	such that $a_i\ge 0$ and $\sum_{i=0}^da_i=k+1$. If $a_i<1$ for all $i$, then $d>k$ and the point $\sum_{i=0}^d(1-a_i)v_i$ is an interior lattice point of $(d-k)P$, a contradiction since $d-k\le d-\deg (P)$. Hence, $a_i\ge 1$ for some $i$, say $a_0\ge 1$. Then \[x=v_0+(a_0-1)v_0+\sum_{i=1}^da_iv_i=v_0+\left((a_0-1)v_0+\sum_{i=1}^da_iv_i\right)\in \mathcal{V}+kP\cap M.\]
	Hence, $k\ge \nu_P$. The conclusion follows.
\end{proof}

\begin{proposition}\label{bound of k_P}
	Let $P=\conv(v_0,\ldots,v_d)$ be a very ample simplex. Then
	\[k_P\le \Vol(P)-|P\cap M|+d+\left\lfloor\frac{d}{2}\right\rfloor.\]
\end{proposition}
\begin{proof}
	Form Proposition \ref{Bound of k_P for very ample simplex}, \eqref{HKN}, and Lemma \ref{nu_P<=deg P-1},
	\begin{align*}
	k_P&\le d_P+\nu_P-1\le d_P+\deg(P)-1\\
	&\le d_P+\Vol(P)-|P\cap M|+d.
	\end{align*}
	By \cite[Proposition 2.2]{Ogata2005}, $d_P\le \frac{d}{2}$. Therefore, since $k_P$, $\Vol(P)$, and $|P\cap M|$ are all integers,
	\[k_P\le \Vol(P)-|P\cap M|+d+\left\lfloor\frac{d}{2}\right\rfloor.\]
\end{proof}

\begin{remark}
	We show some cases that the result of Proposition \ref{bound of k_P} is stronger than \cite[Proposition 2.4]{Ogata2005}:
	\begin{enumerate}
		\item $\Vol(P)\le |P\cap M|+2$. In this case, 
		\[\Vol(P)-|P\cap M|+d+\left\lfloor\frac{d}{2}\right\rfloor\le d+\left\lfloor\frac{d}{2}\right\rfloor-2.\]
		\begin{example}
			Let $\Delta_d$ be the standard $d$-simplex. Then
			
			\[\Vol(\Delta_d)-|\Delta_d\cap M|+d+\left\lfloor\frac{d}{2}\right\rfloor=1-(d+1)+d+\left\lfloor\frac{d}{2}\right\rfloor=\left\lfloor\frac{d}{2}\right\rfloor.\]
			This is clearly a better bound compared to $d+\left\lfloor\frac{d}{2}\right\rfloor-1$.
			
		\end{example}
		\item $P^0\cap M=\emptyset$ or equivalently $\deg(P)\le d-1$. Indeed, in this case, 
		\[k_P\le d_P+\deg(P)-1\le \left\lfloor\frac{d}{2}\right\rfloor+d-2.\]
		We will show in next section that this is the only case that we still need to consider in order to verify the Eisenbud-Goto conjecture for very ample simplices.
		\begin{example}
			Consider $P=2\Delta_d$ for $d\ge 4$, where $\Delta_d$ is the standard $d$-simplex. Then $\deg (P)=2$ and by Proposition \ref{Bound of k_P for very ample simplex}, 
			\[k_P\le d_P+1\le \left\lfloor\frac{d}{2}\right\rfloor+1  <  \left\lfloor\frac{d}{2}\right\rfloor+d-1.\]
		\end{example}
		
	\end{enumerate}
\end{remark}

\begin{theorem}\label{EG-type bound theorem}
	Suppose that $X$ is a fake weighted projective space embedded in $\PP^r$ via a very ample line bundle. Then
	\[\reg(X)\le \deg(X)-\codim(X)+\left\lfloor\frac{\dim(X)}{2}\right\rfloor.\]
\end{theorem}
\begin{proof}
	Let $P$ be the corresponding polytope of the embedding. From \eqref{reg(X) formula}, \eqref{HKN}, and Proposition \ref{bound of k_P}, it follows that 
	\begin{align*}
	\reg(X)\le \Vol(P)-|P\cap M|+d+\left\lfloor\frac{d}{2}\right\rfloor+1= \deg(X)-\codim(X)+\left\lfloor\frac{d}{2}\right\rfloor.
	\end{align*}
\end{proof}

\section{The Eisenbud-Goto Conjecture for Non-hollow Very Ample Simplices}
In this section, we will improve the bound of $k$-normality for non-hollow very ample simplices.

\begin{definition}
	A lattice polytope is hollow if it has no interior lattice points.
\end{definition}
We now show that the inequality \eqref{k_P inequality} holds for non-hollow very ample simplices.
\begin{proposition}\label{EG for simplices}
	Let $P\subseteq M_{\RR}$ be a non-hollow very ample lattice $d$-simplex. Then
	\[k_P\le \Vol(P)-|P\cap M|+d+1.\]
\end{proposition}
\begin{proof}
	We will consider two cases, namely $|P\cap M|=d+2$ and $|P\cap M|\ge d+3$. For the first case, we have the following lemma:
	\begin{lemma}\label{EG for |P|=d+2}
		Suppose that $P=\conv(v_0,\ldots,v_d)$ is a very ample lattice $d$-simplex with $u$ is the only lattice point beside the vertices. Then $P$ is normal.
	\end{lemma}
	\begin{proof}
		Assume that $d_P\ge 2$. Then there exists a point $p\in d_PP\cap M$ such that $p$ cannot be written as $p=x+w$ for some $x\in (d_P-1)P\cap M$ and $w\in P\cap M$. Since $P$ is a simplex, $u$ and $p$  can be uniquely written as
		\[p=\sum_{i=0}^d\lambda_iv_i,\hspace{5mm} \sum_{i=0}^d\lambda_i=d_P\]
		and
		\[u=\sum_{i=0}^d\lambda_i^*v_i,\hspace{5mm}  \sum_{i=0}^d\lambda_i^*=1,\]
		respectively. It follows from the condition of $p$ that $\lambda_i<1$ for all $0\le i\le d$ and there exists $0\le i\le d$ such that $\lambda_i < \lambda_{i}^*$, say $i=0$. By Lemma \ref{Ogata05, Lemma 2.1}, 
		\[p+(d_P-1)v_1=\sum_{i=0}^da_iv_i+bu\]
		for some $a_i, b\in\ZZ_{\ge 0}$ such that $b+\sum_{i=0}^da_i=2d_P-1$. Replacing $p$ by $\sum_{i=0}^d\lambda_iv_i$ and $u$ by $\sum_{i=0}^d\lambda_i^*v_i$ yields
		\begin{align*}
		\lambda_0&=a_0+b\lambda_0^*\\
		\lambda_1+d_P-1&=a_1+b\lambda_1^*\\
		\lambda_2&=a_2+b\lambda_2^*\\
		\hdots\\
		\lambda_d&=a_d+b\lambda_d^*.
		\end{align*}
		Since $\lambda_0<\lambda_0^*$ and $\lambda_i<1$ for all $0\le i\le d$, it follows that $a_0=a_2=\cdots=a_d=0$ and $b=0$. Then $p=d_Pv_1$, a contradiction to the choice of $p$. Therefore, $P$ is normal.
	\end{proof}
	As a consequence, $1=k_P\le \Vol(P)-|P\cap M|+d+1=\Vol(P)-1$. Now we consider the case $|P\cap M|\ge d+3$. By the hypothesis, $|P\cap M|-(d+1)\ge 2$. Consider the Ehrhart vector $h^*=(h_0^*,\cdots, h_d^*)$ of $P$. We have
	\begin{align*}
	h_0^*&=1\\
	h_1^*&=|P\cap M|-d-1\ge 2\\
	h_d^*&=|P^0\cap M|\ge 2.
	\end{align*}
	By \cite[Theorem 1.1]{Hibi1994}, $2\le h_1^*\le h_i^*$ for all $1\le i<d$. Therefore,
	\[\Vol(P)-|P\cap M|+d+1=h_0^*+h_2^*+\cdots+h_d^*\ge 1+2(d-1)=2d-1.\]
	By \cite[Proposition 2.4]{Ogata2005}, 
	\[k_P\le \left\lfloor\frac{d}{2}\right\rfloor+d-1 \le 2d-1 \le \Vol(P)-|P\cap M|+d+1\]
	for all $d\ge 3$. The conclusion follows.
\end{proof}

Let us now recall the definition of Fano polytopes:
\begin{definition}
	A Fano polytope is a convex lattice polytope $P\subseteq M_{\RR}$ such that $P^0\cap M=\{0\}$ and each vertex $v$ of $P$ is a primitive point of $M$.
\end{definition}

From Proposition \ref{EG for simplices}, we obtain the following corollary:

\begin{corollary}
	The Eisenbud-Goto conjecture holds for any projective toric variety corresponding to a very ample Fano simplex.
\end{corollary}

\section{Final Remarks}
We start with a remark that Proposition \ref{Bound of k_P for very ample simplex} fails in general.
\begin{example}[\cite{Gubeladze2009}]\label{Bruns example}
	Consider the polytope $P$ which is the convex hull of the vertices given by the columns of the following matrix
	\[M=\left(
	\begin{array}{cccccccccc}
	0&1&0&0&1&0&1&1\\
	0&0&1&0&0&1&1&1\\
	0&0&0&1&1&1&s&s+1
	\end{array}
	\right)
	\]
	with $s\ge 4$. Then $d_P=\nu_P=2$, and by \cite[Theorem 3.3]{Beck2015}, $k_P=s-1$. It is clear that $k_P>d_P+\nu_P-1$ for all $s\ge 6$.
	
	Furthermore, it can be shown that $P$ cannot be covered by very ample simplicies (\cite[Proposition 4.3.3]{Tran2018}) ; hence, it is very unlikely that we can apply Proposition \ref{Bound of k_P for very ample simplex} to find a bound of the $k$-normality of generic very ample polytopes.
\end{example}

\subsection{What About Hollow Very Ample Simplices}
Finally, we would love to see a classification of hollow very ample lattice simplices. For dimension $2$, Rabinotwiz \cite[Theorem 1]{Rabinowitz1989} showed that any such simplex is unimodularly equivalent to either $T_{p,1}:=\conv(0,(p,0),(0,1))$ for some $p\in \NN$ or $T_{2,2}=\conv(0,(2,0),(0,2))$. Now we will show a way to obtain some hollow very ample simplices in any dimension with arbitrary volume.

We recall the definition of lattice pyramids as in \cite{Nill2008}:
\begin{definition}
	Let $B\subseteq \RR^k$ be a lattice polytope with respect to $\ZZ^k$. Then $\conv(0, B \times \{1\}) \subseteq \RR^{k+1}$ is a lattice polytope with respect to $\ZZ^{k+1}$, called the ($1$-fold) standard pyramid over $B$. Recursively, we define for $l \in \NN_{\ge 1}$ in this way the $l$-fold standard pyramid over $B$. As a convention, the $0$-fold standard pyramid over $B$ is $B$ itself. 
\end{definition}

\begin{proposition}\label{pyr(P) is very ample iff P  is normal}
	Let $P$ be a lattice polytope. Then the 1-fold pyramid over $P$ is very ample if and only if $P$ is normal. 
\end{proposition}
\begin{proof}
	Let $Q=\conv(0,P\times\{1\})$ be the 1-fold pyramid over $P$. Then it is easy to see that if $P$ is normal then so is $Q$. Now suppose that $Q$ is very ample. We have $k_Q\ge k_P$ (\cite[Lemma 4.2.2]{Tran2018}) and each lattice point of $k_QQ\cap M$ sits in $(tP\cap M)\times\{t\}$ for some $0\le t\le k_Q$. In particular, suppose that $(x,t)\in (tP\cap M)\times\{t\}\subseteq k_QQ\cap M$. Then
	\[(x,t)=\sum_{i=1}^t(u_i,1)+(k_Q-t)0\]
	for some $u_i\in P\cap M$. It follows that $x=\sum_{i=1}^tu_i$. Hence, $P$ is $t$-normal for all $k_Q\ge t\ge 1$. Since $k_Q\ge k_P$, it follows that $P$ is normal. The conclusion follows.
\end{proof}

From Proposition \ref{pyr(P) is very ample iff P  is normal}, if we take any $(d-2)$-fold pyramid over either $T_{p,1}$ with $p\in \ZZ_{\ge 1}$ or $T_{2,2}$, which are all normal, then we obtain a hollow normal (hence very ample) $d$-simplex with normalized volume $p$. The Eisenbud-Goto conjecture holds for these.
\begin{example}
We give here an example to demonstrate the case that if $Q$ is very ample but not normal then the $1$-fold pyramid over $Q$ is not very ample. Let $Q$ be the convex polytope given by taking $s=4$ in Example \ref{Bruns example}. Then $Q$ is very ample; however, the $1$-fold pyramid of $Q$, which is given by the convex hull of 
	\[\left(
	\begin{array}{ccccccccccc}
	0&0&1&0&0&1&0&1&1\\
	0&0&0&1&0&0&1&1&1\\
	0&0&0&0&1&1&1&4&5\\
	0&1&1&1&1&1&1&1&1
	\end{array}
	\right),
	\]
	is not very ample.
\end{example}


%
%
\bibliographystyle{amsalpha}
\bibliography{Bibiliography}{}
\bibliographystyle{plain}

\end{document}